\documentclass[11pt,a4paper]{article}
\usepackage[T1]{fontenc}
\usepackage{lmodern}
\usepackage[margin=27mm,top=25mm,bottom=26mm,headheight=14pt]{geometry}
\usepackage{amsmath,amssymb,amsthm,mathtools}
\usepackage{microtype}
\usepackage{enumitem}
\usepackage{graphicx}
\usepackage{booktabs}
\usepackage{fancyhdr}
\usepackage[hidelinks]{hyperref}
\hypersetup{
  pdftitle={Homological lifts of Arnold strangeness},
  pdfauthor={Noboru Ito},
  pdfsubject={Homological lifts of Arnold strangeness}
}
\setlist{itemsep=2pt,topsep=4pt}
\numberwithin{equation}{section}
\newtheorem{theorem}{Theorem}[section]
\newtheorem{proposition}[theorem]{Proposition}
\newtheorem{lemma}[theorem]{Lemma}
\newtheorem{corollary}[theorem]{Corollary}
\newtheorem{maintheorem}{Theorem}

\theoremstyle{definition}
\newtheorem{definition}[theorem]{Definition}

\newtheorem{remark}[theorem]{Remark}
\newtheorem{notation}[theorem]{Notation}
\newtheorem*{unnumberedremark}{Remark}
\DeclareMathOperator{\ind}{ind}
\DeclareMathOperator{\rot}{rot}
\DeclareMathOperator{\rank}{rank}
\DeclareMathOperator{\sign}{sign}
\DeclareMathOperator{\ch}{ch}
\newcommand{\Z}{\mathbb Z}
\newcommand{\St}{\operatorname{St}}
\newcommand{\Stq}{\St_q}
\newcommand{\Stqv}{\St_q^v}
\newcommand{\CE}{\mathcal E}
\newcommand{\CC}{\widetilde C}
\title{Homological lifts of Arnold strangeness}
\author{Noboru Ito}
\date{October 25, 2026}

\begin{document}
\maketitle
\begingroup
\renewcommand{\thefootnote}{}
\endgroup

\begin{abstract}
We construct homological lifts of Arnold's strangeness \(\St\) for long and
closed plane curves. For long curves, we introduce a new \(q\)-variable
quantization polynomial \(\Stq\), an edge state sum, designed for
categorification and construct a
triply graded homology that recovers \(\Stq\). We prove strictness by exhibiting
long curves with the same number of crossings and the same \(\Stq\)
(in particular, the same rotation number and \(\St\)) that are distinguished
by the homology. For
closed curves, we introduce a corresponding new polynomial \(\Stq\) and
construct an edge homology. In this case \(\Stq\) is recovered from a signed
\(q\)-character \(\ch_q\) defined on the canonical cycles of the homology rather
than directly from a graded rank. Again, we exhibit closed curves with the same number of crossings
and the same \(\Stq\) (in particular, the same rotation number and \(\St\))
that are distinguished by the homology.
\end{abstract}

\section*{Introduction}
Arnold's plane-curve invariants are \(J^-\), \(J^+\), and \(\St\)
\cite{Arnold1994}. Their \(q\)-variable quantizations were given, respectively,
by Viro's \(P_C(q)\) \cite{Viro1996}, Lanzat--Polyak's \(I_q(C)\)
\cite{LanzatPolyak2013}, and \cite{Ito2023}.\footnote{In \cite{Ito2023}, the quantization polynomial is
denoted by \(\Stq\) and is defined for long curves, equivalently for closed
curves based on an exterior edge; it corresponds to the vertex state sum.}
In \cite{Ito2026}, homologies corresponding to categorifications of \(J^-\)
and \(J^+\) were constructed whose graded ranks recover the Viro and
Lanzat--Polyak polynomials, and their strictness was established. Through these
categorifications, a homological face-state-sum model realizing Viro's analogy
with face-state-sum formulas for quantum knot polynomials was obtained
\cite[Section~5.3]{Viro1996}. A categorification of \(\St\), however, remained
open.

With categorification in mind, we reconsider the \(q\)-quantization of \(\St\).
For long curves we introduce a new polynomial, the edge state sum \(\Stq\),
different from the vertex state sum of \cite{Ito2023}, and construct a
homological lift of \(\St\) through this polynomial. The polynomial is invariant under self-tangency
perestroikas, and after putting \(q=e^h\), its low-order coefficients recover
the rotation number and \(\St\).

The previous vertex state sum \(\Stqv\) of \cite{Ito2023} was not defined
for closed curves without a prescribed base point. Motivated by
Shumakovitch's formula for classical strangeness, with its base-point
correction \cite{Sh1996}, we introduce a new edge-based polynomial, the edge state sum \(\Stq\),
for unbased closed curves and prove its invariance under self-tangency
perestroikas as well as its base-point independence. After setting \(q=e^h\), the factor \(1/2\) in the quadratic coefficient
of the Maclaurin expansion at \(h=0\) arises directly from the exponential
series. Decategorification of the homology constructed here yields \(\Stq\),
which fully recovers Shumakovitch's formula for \(\St\).
Furthermore, unlike the weighted circle homology used for long curves and their
corresponding exterior-edge based closed curves, we define a variant of the
edge complex used in \cite{Ito2026} for the categorification of curvature
contributions, and construct a homological lift from its canonical cycles. In this sense, to the best of our knowledge, these constructions
provide the first homological categorification of Arnold's strangeness for
generic plane curves.

The paper is organized as follows. Section~\ref{sec:b1} introduces the new
polynomial for long curves, and Sections~\ref{sec:b2}--\ref{sec:b6} construct
its homological lift. The decategorification is stated in
Theorem~\ref{thm:based-lift}, and Section~\ref{sec:b8} proves strictness.
Section~\ref{sec:u1} introduces the closed-curve polynomial, and
Sections~\ref{sec:u2}--\ref{sec:u7} construct its homological lift. The recovery
of \(\Stq\) is stated in Theorem~\ref{thm:unbased-lift}, and
Section~\ref{sec:u8} proves strictness. The principal results, including Theorems~\ref{thm:based-lift} and~\ref{thm:unbased-lift}, are summarized in the following main theorems.

\begin{maintheorem}\label{thm:main-A}
Let \(C\) be an oriented long generic plane curve.
Let \(V(C)\) be its vertex set. For each \(v\in V(C)\), denote the
incoming lower-left and lower-right edges terminating at \(v\) by
\(e_l^v,e_r^v\), respectively. Let
\[
  E(C)=\coprod_{v\in V(C)}\{e_l^v,e_r^v\}
\]
be the set of oriented incoming edges terminating at crossings; the departing
edge at infinity is not included. For \(e\in E(C)\), let \(\ind(e)\) be the
average of the mapping degrees of the two adjacent regions. If \(e_l^v\) is
encountered first, set
\(s(e_l^v)=-1\), \(s(e_r^v)=+1\); if \(e_r^v\) is encountered first, reverse
the signs. Then
\[
  \Stq(C)=\sum_{e\in E(C)}s(e)q^{\ind(e)}
\]
is invariant under self-tangency perestroikas. With \(q=e^h\),
\[
  \St_{e^h}(C)=\rot(C)h+\St(C)h^2+O(h^3),
\]
where \(\rot(C)\) is the rotation number of the long curve and \(\St(C)\) is
the classical strangeness of the corresponding closed curve based on an
exterior edge.

Let \(C^r\) be the unique 2-reduced representative of \(C\) in the sense of
Definition~\ref{def:2reduced} and Proposition~\ref{prop:reduction}, and write
\[
  H^{i,j,k}(C):=H^{i,j,k}(C^r).
\]
There exists a triply graded homology \(H^{i,j,k}(C)\) with graded-rank
polynomial
\[
  \mathcal P(C;t,q,u)
  =\sum_{i,j,k}t^iq^ju^k\,\rank H^{i,j,k}(C)
\]
satisfying
\[
  \left.\frac{\partial}{\partial u}\mathcal P(C;-1,q,u)\right|_{u=1}
  =-\Stq(C).
\]
Moreover, there exist two sequences of long curves
\(\{C_m\}_{m\ge0}\) and \(\{C'_m\}_{m\ge0}\) such that
\[
  \#V(C_m)=\#V(C'_m)=4(m+1),\qquad
  \Stq(C_m)=\Stq(C'_m)=0,
\]
and, writing \(\Stqv\) for the vertex quantization introduced in
\cite{Ito2023},
\[
  \Stqv(C_m)=\Stqv(C'_m)=0.
\]
In particular, the rotation numbers of these long curves and their
classical strangeness values satisfy
\[
  \rot(C_m)=\rot(C'_m)=0,\qquad \St(C_m)=\St(C'_m)=0.
\]
Nevertheless,
\[
  H^{*,*,*}(C_m)\not\cong H^{*,*,*}(C'_m)
\]
as triply graded groups, and
\[
  H^{*,*,*}(C_m)\not\cong H^{*,*,*}(C_n)\qquad(m\ne n).
\]
Thus, for each \(m\), the homology distinguishes \(C_m\) from \(C'_m\),
and it also distinguishes \(C_m\) from \(C_n\) whenever \(m\ne n\),
although these curves have the same \(\Stqv\) and the same \(\Stq\)
(in particular, the same rotation number and \(\St\)).
\end{maintheorem}

\begin{maintheorem}\label{thm:main-B}
Let \(C\) be an oriented generic closed plane curve with no prescribed base
point. Choose a temporary base point \(b\) on an edge \(e^b\), called the
base edge. Let \(E(C)\) be the edges terminating at crossings.
Retain the original base edge (denoted \(e_1^v\) if it terminates at
\(v\)). Introduce \(e_0^b\), for the side entering \(b\), and \(e_1^b\),
for the side leaving \(b\), as in the base-edge convention of
Notation~\ref{not:base-edge}; put
\(E_b(C)=E(C)\sqcup\{e_0^b,e_1^b\}\).
Here \(\ind(e)\) denotes the average of the mapping degrees of the
two regions adjacent to the edge \(e\). Put
\[
 s(e_0^b)=s(e_1^b)=1,\qquad
 \ind(e_0^b)=\ind(e_1^b)=\ind(e^b).
\]
When there are no crossings, \(E(C)=\varnothing\). The signs \(s(e_l^v)\) and \(s(e_r^v)\) of the incoming edges at a
crossing \(v\) are the local intersection numbers: when
\(e_l^v\) is encountered first from \(b\), assign
\(s(e_l^v)=-1\), \(s(e_r^v)=+1\); if \(e_r^v\) is encountered first,
reverse the signs. Then
\[
  \Stq(C)=\sum_{e\in E_b(C)}s(e)q^{\ind(e)}
\]
is independent of \(b\) and invariant under self-tangency perestroikas. With
\(q=e^h\),
\[
  \St_{e^h}(C)
  =2+\rot(C)h+\left(\St(C)+\frac14\right)h^2+O(h^3).
\]

Let \(C^r\) be the 2-reduced representative of the unbased closed
curve \(C\), obtained by repeated crossing-decreasing self-tangency perestroikas and
unique up to orientation-preserving planar isotopy, as stated in
Proposition~\ref{prop:u-unique-reduction}. Let \(V(C^r)\) be the vertex set
of \(C^r\). Write \(\CE_b(C):=\CE_b(C^r)\). There exists an edge complex
\[
  \CE_b(C):\qquad
  0\longrightarrow\mathcal E_b\xrightarrow{\partial}\mathcal V_b
  \longrightarrow0,
\]
where \(\mathcal E_b=\Z\langle E_b(C^r)\rangle\),
\(\mathcal V_b=\Z\langle V(C^r)\sqcup\{b\}\rangle\), and
\[
  \partial[e_l^v]=\partial[e_r^v]=[v],\qquad
  \partial[e_0^b]=\partial[e_1^b]=[b].
\]
Its first homology is independent of the temporary base point up to
isomorphism. On the canonical cycles
\[
 c_v=[e_l^v]-[e_r^v],\qquad c_b=[e_0^b]-[e_1^b],
\]
define a signed \(q\)-character by
\[
 \ch_q(c_v)=s(e_l^v)\bigl(q^{i+1/2}-q^{i-1/2}\bigr)
 \quad(i=\ind(v)),\qquad
 \ch_q(c_b)=2q^{\ind(e^b)}.
\]
Here \(\ind(v)\) is the average of the Alexander numberings of the
four locally adjacent regions at the crossing \(v\).
Then
\[
  \sum_{v\in V(C^r)}\ch_q(c_v)+\ch_q(c_b)=\Stq(C).
\]
Equivalently, the character polynomial \(\mathcal P_{\CE_b}(C;t,q)\)
satisfies
\[
  \mathcal P_{\CE_b}(C;t,q)=t\Stq(C),\qquad
  \mathcal P_{\CE_b}(C;-1,q)=-\Stq(C).
\]
Moreover, there exist two sequences of closed curves
\(\{C_m\}_{m\ge0}\) and \(\{C'_m\}_{m\ge0}\) such that
\[
  \#V(C_m)=\#V(C'_m)=4(m+1),\qquad
  \Stq(C_m)=\Stq(C'_m)=2q^{1/2},
\]
For these closed curves, the rotation numbers and classical
strangeness values satisfy
\[
  \rot(C_m)=\rot(C'_m)=1,\qquad \St(C_m)=\St(C'_m)=0.
\]
Nevertheless,
\[
  H_1(\CE_b(C_m))\not\cong H_1(\CE_b(C'_m)),
\]
and
\[
  H_1(\CE_b(C_m))\not\cong H_1(\CE_b(C_n))\qquad(m\ne n).
\]
Thus, for each \(m\), the edge homology distinguishes \(C_m\) from \(C'_m\),
and it also distinguishes \(C_m\) from \(C_n\) whenever \(m\ne n\),
although these curves have the same \(\Stq\)
(in particular, the same rotation number and \(\St\)).
\end{maintheorem}

Khovanov homology \cite{Khovanov2000} is a categorification of the Jones
polynomial, whose Maclaurin expansion after setting \(q=e^h\) has Vassiliev
invariants as its coefficients. Viewing the categorification of Arnold's
order-one invariants in light of this structure suggests quantization
polynomials whose low-order Maclaurin coefficients recover Arnold invariants.
Polynomials encoding \(\St\) in their low-order Taylor coefficients have been
constructed, for example, in \cite{Ito2023,CTY2026}.
In this paper, we define quantization polynomials as edge state sums suited to
categorification for based and unbased curves, respectively, and \emph{categorify}
them. As shown in Theorems~\ref{thm:main-A} and~\ref{thm:main-B}, the homologies
corresponding to the categorifications in the based and unbased settings,
respectively, have greater distinguishing power than the quantization
polynomials obtained by decategorification.

Theorem~\ref{thm:main-A} follows from Theorem~\ref{thm:based-lift} and the
calculations in Section~\ref{sec:b8}. Theorem~\ref{thm:main-B} follows from
Theorem~\ref{thm:unbased-lift} and the calculations in
Section~\ref{sec:u8}.

All modules are over \(\Z\). In both settings of main theorems (A) and (B),
\(\Stq(C)\in\Z[q^{1/2},q^{-1/2}]\); the expansions after \(q=e^h\)
are Maclaurin expansions at \(h=0\). In this paper, we write \(\Stq\) for the two new edge state sums for based and unbased curves, respectively, and \(\Stqv\) for the vertex quantization denoted \(\Stq\)
in \cite{Ito2023}; the latter is used only in the based comparison.

\section{Long curves and exterior-edge based closed curves}\label{sec:based}

A based curve in this section is an oriented generic closed plane curve
whose base point lies on an \emph{exterior edge}, namely an edge incident
to the unbounded region. It is identified with the corresponding long curve;
arbitrary based closed curves are not considered in this section. After an
orientation-preserving planar isotopy, assume that the two ends of the long
curve agree with the \(x\)-axis outside a compact set. We write
\(C_{\mathrm{long}}\) for this long curve and \(C_{\mathrm{closed}}\) for the
closed curve obtained by joining the negative and positive ends \emph{above} the
\(x\)-axis. With this convention,
\begin{equation}\label{eq:closed-long-rot}
 \rot(C_{\mathrm{closed}})=\rot(C_{\mathrm{long}})+1.
\end{equation}
We use this convention to fix the one-to-one correspondence between the
oriented long and exterior-edge based closed settings.
Throughout this section, \(\rot(C)\) denotes the rotation number of the long
curve.

For a long curve \(C\) with \(n\ge1\) crossings, number the oriented
edges in traversal order from the end extending to \(-\infty\):
\(e_1,e_2,\ldots,e_{2n},e_{2n+1}\), where
\(e_{2n+1}=e_\infty\) is the unique edge extending to \(+\infty\).
The first and last edges form the pair of unbounded ends. An edge
terminating at a crossing \(v\) may also be denoted by \(e_i^v\);
when recording its local left/right position, we write \(e_l^v,e_r^v\).
For the vertex set \(V(C)\) of \(C\), the edge set used in the polynomial is
\begin{equation}\label{eq:incoming-set}
 E(C)=\coprod_{v\in V(C)}\{e_l^v,e_r^v\}.
\end{equation}
It includes the incoming edge from infinity when that edge terminates at a
crossing. The departing edge \(e_\infty\), which has no terminal crossing,
is excluded. When \(n=0\), the long curve has one edge extending to
both ends and \(E(C)=\varnothing\).
Figure~\ref{fig:incoming-edges} illustrates this convention.

\begin{figure}[htbp]
\centering
\includegraphics[width=0.84\linewidth]{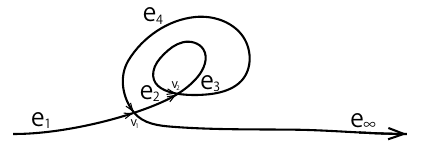}
\caption{Incoming edges of a long curve. The four edges
\(e_1,e_2,e_3,e_4\) belong to \(E(C)\), whereas \(e_\infty\notin E(C)\).}
\label{fig:incoming-edges}
\end{figure}

\subsection{The quantization polynomial \(\Stq\) and its low-order coefficients}\label{sec:b1}

Let \(C\) be an oriented long generic plane curve corresponding to an
exterior-edge based closed curve.
\begin{definition}[The polynomial \(\Stq\)]\label{def:based-Stq}
For the edge set \(E(C)\) in \eqref{eq:incoming-set}, the edge signs
\(s(e)\) (Definition~\ref{def:based-sign}) and Alexander numbering
\(\ind(e)\) (Definition~\ref{def:edge-index}), define
\[
  \Stq(C)=\sum_{e\in E(C)} s(e)q^{\ind(e)}.
\]
\end{definition}
Putting \(q=e^h\) gives
\[
  \St_{e^h}(C)=\sum_{e\in E(C)} s(e)e^{h\ind(e)},\qquad
  e^{h\ind(e)}
  =1+\ind(e)h+\tfrac12\ind(e)^2h^2+\cdots .
\]
The following identities hold:
\begin{align*}
  \sum_{e\in E(C)} s(e)&=0, \\
  \sum_{e\in E(C)} s(e)\ind(e)&=\rot(C), \\
  \frac12\sum_{e\in E(C)} s(e)\ind(e)^2&=\St(C).
\end{align*}
For completeness, we explain the rotation-number identity. Regard \(C\)
as the corresponding closed curve with a base point \(b\) on the exterior
edge \(e^b\). Here \(i=\ind(v)\) is the average of the Alexander
numberings of the four locally adjacent regions at \(v\). At each
crossing \(v\) of index \(i\), the two incoming edges have indices \(i+\tfrac12\) and \(i-\tfrac12\). By the prescribed
intersection-number signs, their combined contribution to the first
moment is either
\((i+\tfrac12)-(i-\tfrac12)\) or
\((i-\tfrac12)-(i+\tfrac12)\), hence equals
\(\sign(v,b)=-\operatorname{sgn}\det(T_1,T_2)\), where \(T_1,T_2\) are
the oriented tangents on the first and second visits to \(v\) from \(b\).
Consequently,
\[
 \sum_{e\in E(C)}s(e)\ind(e)=\sum_{v\in V(C)}\sign(v,b).
\]
Whitney's formula \cite[Theorem~1.6.17]{Fukui} (cf.\ \cite[Section~1.B, p.~451]{EkholmWolff2002}, with the opposite crossing-sign convention) gives the rotation
number of the corresponding closed curve as
\(2\ind(e^b)+\sum_v\sign(v,b)\). In the present upper closure
\(\ind(e^b)=\tfrac12\), and
\(\rot(C_{\mathrm{closed}})=\rot(C_{\mathrm{long}})+1\) by
\eqref{eq:closed-long-rot}; thus the displayed sum equals the long
rotation number \(\rot(C)\). The quadratic identity follows from
Shumakovitch's formula \cite[Section~1.5]{Sh1996}: for the exterior
base edge with \(\ind(e^b)=\tfrac12\), its correction
\(\ind(e^b)^2-\tfrac14\) vanishes.
Thus the constant, linear, and quadratic coefficients are \(0\),
\(\rot(C)\), and \(\St(C)\), respectively. Hence \(\Stq\) is a \(q\)-variable quantization of \(\St\).

For the based comparison, the vertex polynomial is given by
\begin{equation}\label{eq:vertex-poly}
 \Stqv(C)=\sum_{v\in V(C)}w(v)q^{\ind(v)},\qquad w(v)=s(e_l^v).
\end{equation}
It is used only for the comparison in Section~\ref{sec:b8}.

\subsection{Invariance and reduction under self-tangency perestroikas}\label{sec:b2}

\begin{proposition}\label{prop:based-II}
Let \(C\) be an oriented long generic plane curve. Then \(\Stq(C)\) is
unchanged by a self-tangency perestroika.
\end{proposition}
\begin{proof}
The two crossings of the 2-gon have the same index \(i\) and opposite
signs. The contributions of their four incoming edges are
\[
  q^{i+1/2}-q^{i-1/2}
  \quad\text{and}\quad
  -q^{i+1/2}+q^{i-1/2}.
\]
Their sum is zero.
\end{proof}

Two long curves are equivalent if and only if representatives are related
by a finite sequence of planar isotopies and self-tangency perestroikas
supported in compact subsets of the plane, with both ends along the
positive and negative directions of the \(x\)-axis fixed at infinity.
Here self-tangency perestroikas are local moves creating or removing
\emph{bounded} 2-gons. All reductions of long curves are taken up to
orientation-preserving planar isotopy, with the two ends fixed outside
a compact set.

\begin{definition}[Polygons and 2-reduced long curves]\label{def:2reduced}
Let \(C\) be a long generic plane curve and let \(R\) be a complementary
region of \(\mathbb R^2\setminus C\). If
\(\partial R\setminus V(C)\) has exactly \(n\) connected components, then
\(R\) is called an \(n\)-gon. The unbounded complementary region is treated
by the same definition. A long curve is called \emph{2-reduced} if it has no
bounded 2-gon. A 2-reduced long curve obtained from \(C\) by repeatedly applying
crossing-decreasing self-tangency perestroikas is denoted by \(C^r\).
\end{definition}

\begin{proposition}[Uniqueness of 2-reduction]\label{prop:reduction}
Let \(C\) be a long generic plane curve. Its 2-reduced representative \(C^r\)
is unique up to orientation-preserving planar isotopy with the two ends fixed
outside a compact set.
\end{proposition}
\begin{proof}
Apply Lemma~\ref{lem:common-reduction} to any two reduced representatives.
A nonempty increasing sequence ends by creating a 2-gon, so it cannot
end at a reduced curve. Both representatives therefore coincide with
their common reduction up to the specified isotopy.
\end{proof}

\begin{lemma}[Common 2-reduction]\label{lem:common-reduction}
Let \(C\) and \(C'\) be long curves. If they are related by self-tangency
perestroikas, they have a common 2-reduced curve from which both are obtained
using only crossing-increasing self-tangency perestroikas.
\end{lemma}
\begin{proof}
Write \(2a\) for an increasing self-tangency perestroika and \(2b\) for a
decreasing one. We distinguish the local configurations in which bounded
2-gons are created or removed by self-tangency perestroikas.
\begin{enumerate}[label=\textup{(\arabic*)}]
\item If sufficiently small disks containing the respective 2-gons are
disjoint, the two self-tangency perestroikas commute immediately.
\item If the 2-gons created or removed by the respective self-tangency
perestroikas coincide, the mutually inverse increasing/decreasing pair cancels.
\item If the two 2-gons created or removed by the respective self-tangency
perestroikas share one point, the standard local movie cancels the increasing
perestroika followed by the decreasing one.
\end{enumerate}
The creation or removal of an unbounded 2-gon is not included in the
equivalence relation for long curves, so these cases are exhaustive.

These local replacements are repeated to put all decreasing self-tangency
perestroikas before all increasing self-tangency perestroikas.
If necessary, reduce the intermediate curve further; reversing these
additional reductions gives the required increasing sequences.

For comparison, Ito--Takimura~\cite{IT2013} call the 2-gon creation and
removal operations \(2a\) and \(2b\), respectively (H2 moves); Khovanov
\cite{Khovanov1997} uses the names \(+2\) and \(-2\) moves in the doodle setting.
\end{proof}

\begin{corollary}\label{cor:complete-reduction}
For long curves, the 2-reduced representative is a complete invariant
for the equivalence relation generated by self-tangency perestroikas. Moreover,
\[
  \Stq(C)=\Stq(C^r).
\]
\end{corollary}
\begin{proof}
The first assertion follows from Lemma~\ref{lem:common-reduction}: curves with
the same reduced representative are related through it by self-tangency perestroikas. The second follows from the invariance of \(\Stq(C)\)
under self-tangency perestroikas
(Proposition~\ref{prop:based-II}).
\end{proof}

\subsection{States and edge signs}\label{sec:b3}

Let \(C^r\) be an oriented long 2-reduced curve. Equivalently, one may,
as set up at the beginning of Section~\ref{sec:based}, regard \(C^r\) as the corresponding closed curve with a base point on an
exterior edge.
\begin{definition}[Edge signs]\label{def:based-sign}
Write \(e_l^v,e_r^v\) for the two incoming edges at a crossing \(v\). Traverse
the long curve in its orientation and mark the crossings according to the order
of the two visits. The associated edge signs are
\[
\begin{array}{c|cc}
  \text{order of visits} & s(e_l^v) & s(e_r^v)\\ \hline
  e_l^v\ \text{first},\ e_r^v\ \text{second} & -1 & +1\\
  e_r^v\ \text{first},\ e_l^v\ \text{second} & +1 & -1
\end{array}
\]
and hence
\[
  s(e_l^v)=-s(e_r^v).
\]
An oriented edge with terminal crossing \(v\) acquires its sign uniquely at
that crossing. We call this sign its \emph{intersection number}.
\end{definition}

Use the correspondence between an oriented long curve and the oriented
closed curve based on an exterior edge, with
\(\rot(C_{\mathrm{closed}})=\rot(C_{\mathrm{long}})+1\).
\begin{definition}[State]\label{def:state}
Let \(C\) be an oriented generic long curve or an oriented generic
closed curve with a chosen base point \(b\). Traverse \(C\) in its
orientation, from \(-\infty\) in the long case and from \(b\) in the
based closed case. The order of the two visits determines the intersection
number at each crossing \(v\); in the based closed case, this is
\(\sign(v,b)\). The \emph{state} is the map
\[
 S:V(C)\longrightarrow\{0,1\}
\]
uniquely determined by requiring \((-1)^{S(v)}\) to equal this
intersection number at every crossing \(v\).
\end{definition}
Recall that, as specified earlier in Section~\ref{sec:based}, we have fixed a
natural one-to-one correspondence between long curves and closed curves
based on exterior edges. Applying oriented smoothing at all crossings of
the based closed curve corresponding to a long curve yields the oriented
Seifert circles \(\CC_n\). The upper-right local diagrams in Figure~\ref{fig:long-marked-state}
represent the pair consisting of these circles and the state \(S\):
they display the crossing marking and the signs of the incident edges,
including the wide/single-edge data. Forgetting \(S\) leaves the
Seifert circles. Furthermore, subdividing each \(\CC_n\) at the smoothing
points gives a correspondence between each original edge \(e_i\) and an arc
on one of the Seifert circles \(\CC_n\) \cite[Notation~1]{Ito2023}.
However, on the circle containing the base point \(b\), we also include
\(b\) as a subdivision point. We denote this correspondence by
\[
  e_i\longmapsto e^{(i)}.
\]
Each \(\CC_n\) can be written as a union of these arcs.
By abuse of notation, we also write \(e_i\) for the corresponding
arcs \(e^{(i)}\).
Figure~\ref{fig:long-marked-state} also shows the long
curve and its corresponding based closed realization.

\begin{figure}[p]
\centering
\includegraphics[width=\linewidth,height=0.82\textheight,keepaspectratio]{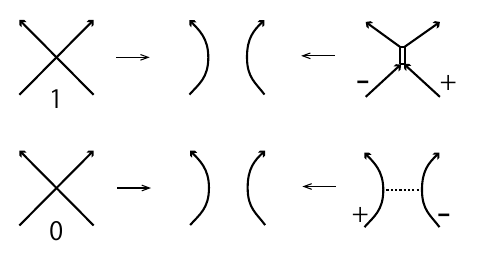}
\includegraphics[width=\linewidth]{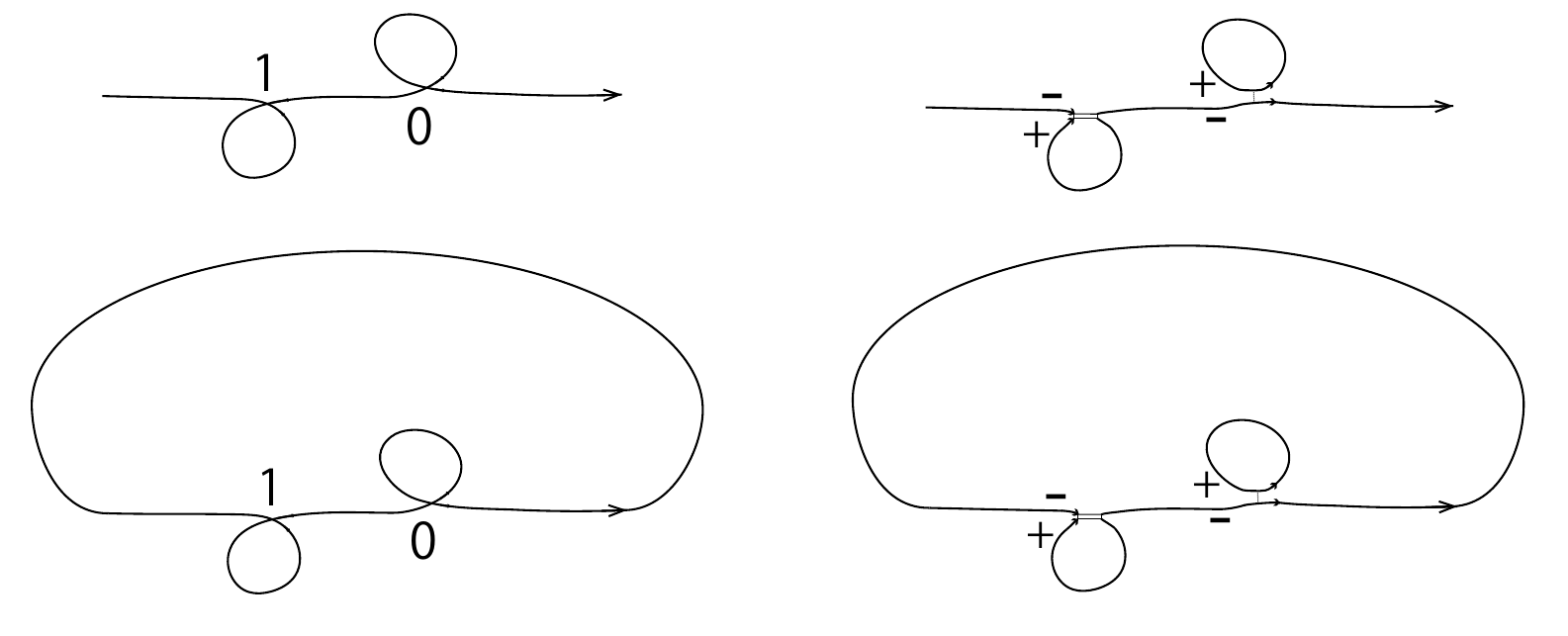}
\caption{The oriented smoothing and a diagrammatic representation of
Seifert circles together with the state (top); a long curve and its state
(middle); and the corresponding exterior-edge based closed realization
(bottom), with \(\rot(C_{\mathrm{closed}})=\rot(C_{\mathrm{long}})+1\).
The displayed local signs and labels are retained.}
\label{fig:long-marked-state}
\end{figure}

For comparison, the corresponding exterior-edge based closed realization is
shown in Figure~\ref{fig:based-marked-state}.
\begin{figure}[htbp]
\centering
\includegraphics[width=8cm]{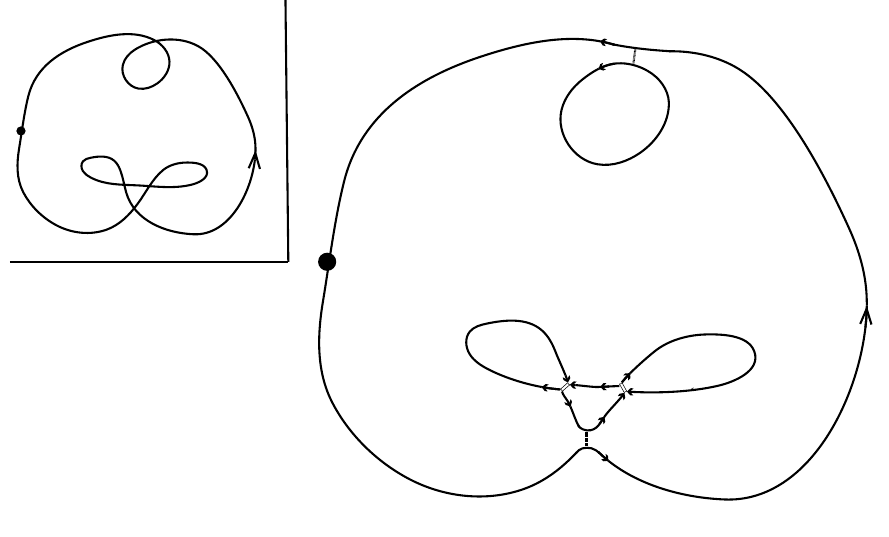}
\caption{The exterior-edge based closed realization of the state. The
base point and the local signs determine the same state as in the long
picture.}
\label{fig:based-marked-state}
\end{figure}

\subsection{Alexander numbering}\label{sec:b4}

\begin{definition}[Alexander index]\label{def:edge-index}
The Alexander index of an edge is the average of the mapping degrees of the
two adjacent regions. Write
\[
  j=\ind(e).
\]
\end{definition}
Let \(\Z\langle E(C^r)\rangle\) be the free abelian group generated by the
edge set \(E(C^r)\), and for \(e\in E(C^r)\) let \([e]\) denote the
corresponding basis generator. The signed edge element can then be grouped by
index:
\[
  \sum_{e\in E(C^r)}s(e)[e]
  =\sum_j\ \sum_{\substack{e\in E(C^r)\\\ind(e)=j}}s(e)[e].
\]

\subsection{Recovery of the polynomial, rotation, and strangeness}
\label{sec:b5}

The edge evaluation \([e]\mapsto q^{\ind(e)}\) gives
\[
  \sum_{e\in E(C^r)} s(e)[e]\longmapsto
  \sum_{e\in E(C^r)} s(e)q^{\ind(e)}
  =\Stq(C^r)=\Stq(C).
\]
The evaluations
\[
  [e]\longmapsto\ind(e),\qquad
  [e]\longmapsto\tfrac12\ind(e)^2
\]
recover \(\rot(C)\) and \(\St(C)\) by the identities from
Section~\ref{sec:b1}:
\[
  \sum_{e\in E(C^r)}s(e)\ind(e)=\rot(C),\qquad
  \frac12\sum_{e\in E(C^r)}s(e)\ind(e)^2=\St(C).
\]

\begin{definition}[Circle weight]\label{def:circle-weight}
For each Seifert circle \(\CC_n\), let
\[
 E_n=\{e\in E(C^r)\mid e\text{ corresponds to an arc on }\CC_n\}.
\]
Its weight is
\[
 k(\CC_n):=\sum_{e\in E_n}s(e).
\]
\end{definition}

\subsection{Weighted circle homology}\label{sec:b6}

The input throughout this construction is \(C^r\). We use the
corresponding exterior-edge based closed curve fixed in
\eqref{eq:closed-long-rot}; in particular, no second closure convention is
introduced. The edge \(e_\infty\) is excluded from \(E(C^r)\).

For each Seifert circle \(\CC_n\), use the weight of
Definition~\ref{def:circle-weight}. The edge
index is constant on \(E_n\), and we set
\[
 j_n=j(\CC_n)=\ind(e)\quad(e\in E_n),\qquad
 k_n=k(\CC_n).
\]

In the crossing-free special case, via the CW decomposition described
below, we set \(j=1/2\) and \(k=0\).
Give each circle the CW decomposition induced by this subdivision, with the
base point \(b\) a 0-cell on the circle containing it from the outset.
We also allow the minimal decomposition
with one 0-cell and one loop 1-cell.

Orient the 1-cells consistently with
\(\CC_n\); their sum represents the fundamental class
\([\CC_n]\in H_1(\CC_n;\Z)\).
Following the fundamental-class model of \cite{Ito2026}, we use
only these fundamental classes as generators and place them in homological
degree \(1\):
\begin{equation}\label{eq:circle-chain}
 C^{i,j,k}(C^r)=
 \begin{cases}
 \displaystyle\bigoplus_{\substack{n:\ j_n=j\\k_n=k}}
      \Z\langle[\CC_n]\rangle,&i=1,\\[3pt]
 0,&i\ne1.
 \end{cases}
\end{equation}
Define the differential \(d=0\) on these generators and set
\[
 H^{i,j,k}(C^r)=H^i\bigl(C^{*,j,k}(C^r),d\bigr).
\]
We call \(H^{i,j,k}(C^r)\) the \emph{weighted circle homology}.
As a special case, for the crossing-free based representative, the single
circle has \((i,j,k)=(1,\tfrac12,0)\), with its CW decomposition based
at the point inherited from the corresponding based closed curve.
\begin{unnumberedremark}[The fundamental-class model]
This model retains the fundamental classes in degree \(1\), but not the
usual \(H_0(\CC_n;\Z)\cong\Z\). The differential \(d=0\) is the
one defined on the fundamental-class generators; it is not the cellular
boundary map of the chosen CW decomposition.
\end{unnumberedremark}

Retain the full graded-rank polynomial
\[
 \mathcal P(C^r;t,q,u)=\sum_{i,j,k}t^iq^ju^k\,\rank H^{i,j,k}(C^r).
\]
Since every generator is in degree \(1\),
\begin{equation}\label{eq:circle-D}
 \left.\frac{\partial}{\partial u}\mathcal P(C^r;t,q,u)\right|_{u=1}
 =t\sum_n k_nq^{j_n}.
\end{equation}
The signed terms regroup as
\begin{equation}\label{eq:circle-regroup}
 \sum_n k_nq^{j_n}
 =\sum_n\sum_{e\in E_n}s(e)q^{\ind(e)}
 =\sum_{e\in E(C^r)}s(e)q^{\ind(e)}=\Stq(C).
\end{equation}

\begin{theorem}[Based homological lift]\label{thm:based-lift}
The triply graded isomorphism type of \(H^{i,j,k}(C^r)\) is an invariant of
long curves up to self-tangency perestroikas. The principal \(t=-1\)
specialization of its graded-rank polynomial satisfies
\[
 \left.\frac{\partial}{\partial u}\mathcal P(C^r;-1,q,u)\right|_{u=1}
 =-\Stq(C).
\]
\end{theorem}
\begin{remark}[The specialization at \(t=1\)]\label{rem:based-t1}
The specialization
\[
 \left.\frac{\partial}{\partial u}\mathcal P(C^r;1,q,u)\right|_{u=1}
 =\Stq(C)
\]
recovers \(\Stq(C)\) without the overall homological sign.
\end{remark}
\begin{proof}[Proof of Theorem~\ref{thm:based-lift}]
The uniqueness of \(C^r\) (Proposition~\ref{prop:reduction}) identifies the oriented Seifert circles and preserves
\(j\) and \(k\). Equations~\eqref{eq:circle-D} and
\eqref{eq:circle-regroup} give the displayed specialization.
\end{proof}
\begin{remark}[Sign conventions]\label{rem:based-signs}
For an exterior-edge based curve, let \(T_1,T_2\) be the oriented
tangent vectors on the first and second visits to a crossing \(v\) from
the base point \(b\). The intersection-number convention assigns
\(\sign(v,b)=-\operatorname{sgn}\det(T_1,T_2)\), opposite to the
ordered determinant convention. Since the generators in
\eqref{eq:circle-chain} have homological degree \(1\), the
\(t=-1\) decategorification also produces an overall minus sign.
\end{remark}

\begin{remark}[An edge-homological lift in the based case]\label{rem:based-edge}
Via the correspondence between oriented long curves and exterior-edge
based oriented closed curves, there is also a natural edge-homological
construction in the based case. We include its sketch because it is
parallel to the unbased closed-curve construction of Section~\ref{sec:u7}.
It does not have the Alexander-index grading \(j\) and edge-weight grading \(k\) of the weighted circle homology,
and is therefore secondary to the triply graded construction here.
Nevertheless, its direct geometric relation to the closed case makes
this supplementary construction useful. Define the edge complex
\[
  \CE(C):\qquad 0\longrightarrow \Z\langle E(C)\rangle
  \xrightarrow{\partial}\Z\langle V(C)\rangle
  \longrightarrow0,
  \qquad \partial(e^v)=v,
\]
where \(e^v\) denotes an edge with terminal crossing \(v\). The edge and
vertex groups lie in degrees \(1\) and \(0\). For each crossing,
\[
  c_v=[e_l^v]-[e_r^v],\qquad z_v=s(e_l^v)c_v
\]
are a canonical cycle and a signed canonical cycle, respectively. Writing
\(i=\ind(v)\), the local index convention is
\[
  \ind(e_l^v)=i+\tfrac12,\qquad
  \ind(e_r^v)=i-\tfrac12.
\]
Define a signed character on the canonical cycle basis by
\[
  \ch_q(c_v)=s(e_l^v)\bigl(q^{i+1/2}-q^{i-1/2}\bigr)=s(e_l^v)q^{\ind(e_l^v)}+s(e_r^v)q^{\ind(e_r^v)}.
\]
For this edge complex,
\[
  H_0(\CE(C))=0,\qquad
  H_1(\CE(C))=\ker\partial=\Z\langle c_v\rangle .
\]
With \(\ch_q H_1(\CE(C)):=\sum_v\ch_q(c_v)\), one obtains
\[
  \ch_q H_1(\CE(C))=\Stq(C).
\]
Regrouping the edge terms by Seifert circle yields \(\sum_n k_nq^{j_n}\).
This is a regrouping of character terms, \emph{not} a \((j,k)\)-grading
decomposition of \(H_1\). Its character polynomial is
\[
  \mathcal P_{\CE}(C;t,q)
  =\sum_i t^i\ch_q H_i(\CE(C))=t\Stq(C).
\]
Consequently its values at \(t=1\) and \(t=-1\) are \(\Stq(C)\) and
\(-\Stq(C)\). For a crossing-free long curve, both modules in this supplementary
construction are zero. In the weighted circle homology, by contrast, the corresponding based closed curve gives one Seifert circle whose
fundamental class is retained as a generator.
\end{remark}

\subsection{Strictness and the shared connected-sum family}
\label{sec:b8}

The invariant compared here is the triply graded isomorphism type of
\(H^{i,j,k}(C^r)\). The aim is to exhibit information gained through categorification, rather than to classify long curves
up to self-tangency perestroikas.
In Figure~\ref{fig:four-crossing-pair}(a), let \(C_1\) be the left
four-crossing curve and \(C_2\) the right one (from
Arnold~\cite[p.~15, Fig.~12]{Arnold1994}). The comparison is
\begin{align*}
 \#V(C_1)&=\#V(C_2)=4,&\rot(C_1)&=\rot(C_2)=0,\\
 \St(C_1)&=\St(C_2)=0,&\Stq(C_1)&=\Stq(C_2),\\
 \Stqv(C_1)&=\Stqv(C_2),\\
 H^{*,*,*}(C_1^r)&\not\cong H^{*,*,*}(C_2^r)
 &&\text{as triply graded groups}.
\end{align*}
The reductions in part (b)
have four and zero crossings, respectively. Their based values are
\begin{equation}\label{eq:based-pair-values}
 \Stq(C_1)=\Stq(C_2)=0,\qquad \Stqv(C_1)=\Stqv(C_2)=0.
\end{equation}

Using the two four-crossing curves \(C_1\) and \(C_2\) as seeds, define
\begin{align*}
 \Gamma_0&=C_1,&\Gamma_{m+1}&=\Gamma_m\#C_1,\\
 K_0&=C_2,&K_{m+1}&=K_m\#C_2\qquad(m\geq0).
\end{align*}
The comparison for these long curves is as follows:
\begin{align*}
 \#V(\Gamma_m)&=\#V(K_m)=4(m+1),&
 \rot(\Gamma_m)&=\rot(K_m)=0,\\
 \St(\Gamma_m)&=\St(K_m)=0,&
 \Stq(\Gamma_m)&=\Stq(K_m)=0,\\
 \Stqv(\Gamma_m)&=\Stqv(K_m)=0,\\
 H^{*,*,*}(\Gamma_m^r)&\not\cong H^{*,*,*}(K_m^r)\qquad(m\geq0).
\end{align*}
Here \(m\) counts connected-sum operations, so each curve has
\(m+1\) factors. The upper sequence is shown in
Figure~\ref{fig:shared-family}; the indicated exterior arcs are joined
compatibly with their orientations and the exterior base point is retained.
Its displayed reduced representative has \(4(m+1)\) crossings.
The lower sequence \(K_m\) reduces to a single circle.
Figure~\ref{fig:state-family} shows, for each factor, two circles of index
\(3/2\), one with signed sum \(+1\) and one with signed sum \(-1\), and
two non-outer circles of index \(1/2\), each with signed sum \(+1\).
The entire connected sum has one outer circle of index \(1/2\); each factor
contributes \(-2\) to its signed sum. Hence the complete graded ranks, for \(m\geq0\), are
\begin{center}
\begin{tabular}{lcccc}
\toprule
Circle type & \(i\) & \(j\) & \(k\) & Rank\\
\midrule
\(j=3/2,\ k=1\) & 1 & \(3/2\) & \(1\) & \(m+1\)\\
\(j=3/2,\ k=-1\) & 1 & \(3/2\) & \(-1\) & \(m+1\)\\
Non-outer \(j=1/2\) & 1 & \(1/2\) & \(1\) & \(2(m+1)\)\\
Outer \(j=1/2\) & 1 & \(1/2\) & \(-2(m+1)\) & \(1\)\\
\bottomrule
\end{tabular}
\end{center}
All unlisted graded parts are zero. In particular,
\[
 \mathcal P(\Gamma_m^r;t,q,u)=t\bigl[(m+1)q^{3/2}(u+u^{-1})+q^{1/2}\{2(m+1)u+u^{-2(m+1)}\}\bigr].
\]
The total rank is \(4m+5\).
Differentiating at \(u=1\) cancels the two \(j=3/2\) contributions and
the non-outer/outer \(j=1/2\) contributions separately. At every factor, the four crossings have index \(1\) and vertex weights
\(+1,+1,-1,-1\). Hence the vertex quantization \(\Stqv\) of \cite{Ito2023}
(for the definition, see \eqref{eq:vertex-poly}) satisfies
\[
 \Stqv(\Gamma_m)=(m+1)(q+q-q-q)=0.
\]
Thus
\[
 \Stq(\Gamma_m)=\Stqv(\Gamma_m)=0\qquad(m\geq0).
\]
For \(C_2^r\), the only graded group is
\(H^{1,1/2,0}(C_2^r)\cong\Z\), and
\(\mathcal P(C_2^r;t,q,u)=tq^{1/2}\). This Poincar\'e polynomial
is nonzero but independent of \(u\), so
\(\partial\mathcal P(C_2^r;t,q,u)/\partial u=0\). Hence the
recovery formula gives \(\Stq(C_2)=0\). The vertex quantization
\(\Stqv(C_2)\) vanishes because the 2-reduced long representative has
no crossings. By invariance under self-tangency perestroikas,
these values also hold before reduction,
giving \eqref{eq:based-pair-values}. Thus, the lower sequence of Figure~\ref{fig:shared-family} also satisfies
\[
 \Stq(K_m)=\Stqv(K_m)=0\qquad(m\geq0),
\]
since \(K_m^r\) is plane isotopic to \(C_2^r\) for any \(m\).

\emph{Nevertheless}, the graded ranks distinguish the pair and all
members of \(\{\Gamma_m\}_{m\geq0}\), proving the asserted strictness.

\begin{figure}[htbp]
\centering
\includegraphics[width=\linewidth]{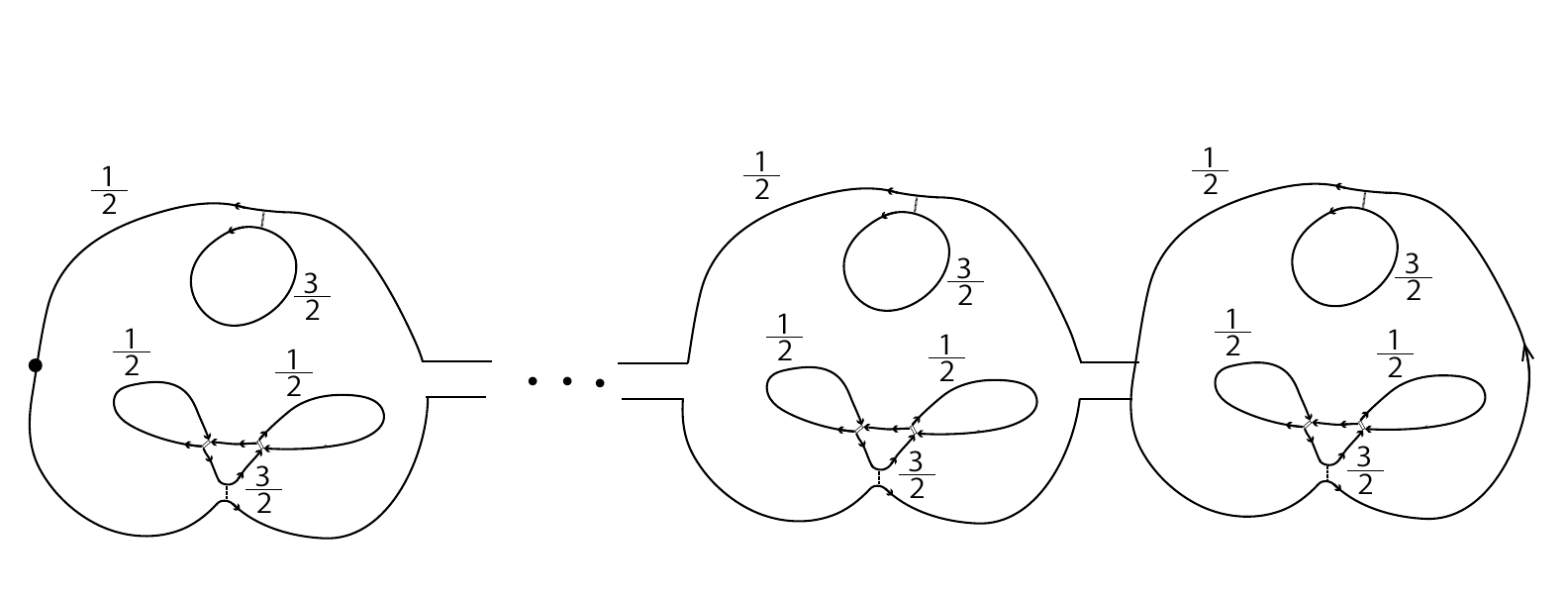}
\caption{The state of the curve \(\Gamma_m\) obtained by connected sums. Each factor
has two index-\(3/2\) circles and two non-outer index-\(1/2\) circles,
and the entire connected sum has one outer index-\(1/2\) circle. The circles used in the calculation are identified by their index and by
outer/non-outer type.}
\label{fig:state-family}
\end{figure}

\begin{figure}[htbp]
\centering
\includegraphics[width=0.94\linewidth]{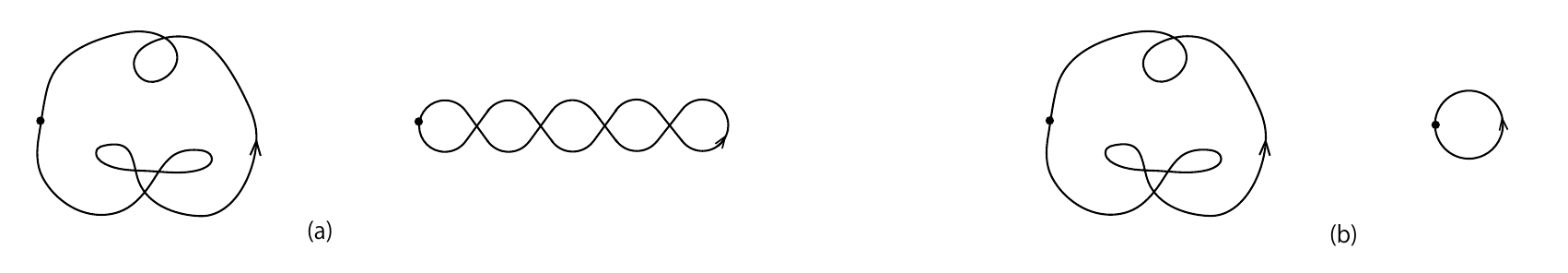}
\caption{The shared four-crossing pair (a) and its 2-reductions
(b), with \(C_1\) on the left and \(C_2\) on the right. The reduced
crossing counts are \(4\) and \(0\). Black points mark exterior base
points in the based case and temporary base points in the unbased case.}
\label{fig:four-crossing-pair}
\end{figure}

\begin{figure}[htbp]
\centering
\textbf{(a)}\par\smallskip
\includegraphics[width=\linewidth]{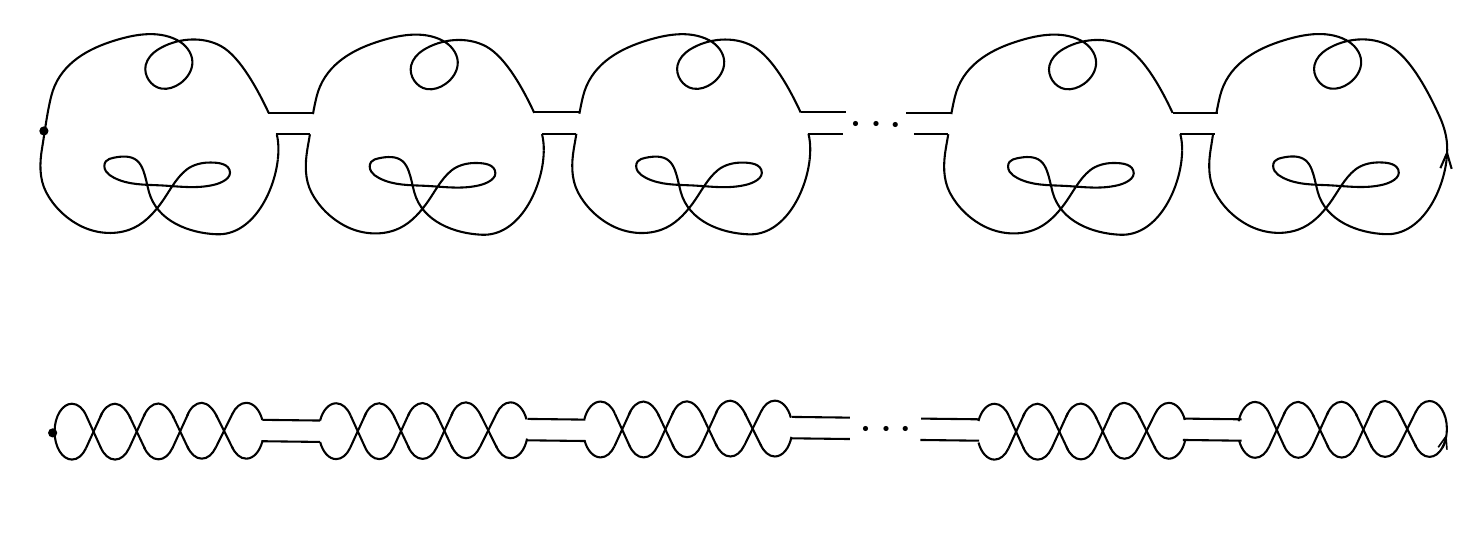}
\par\medskip\textbf{(b)}\par\smallskip
\includegraphics[width=\linewidth]{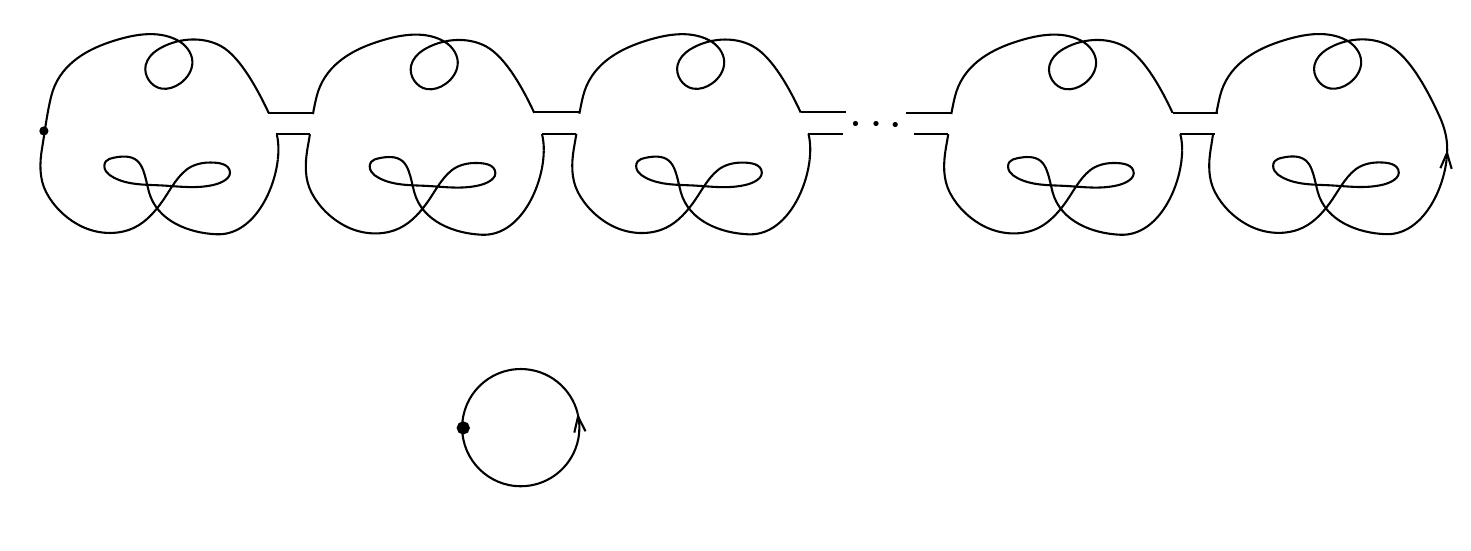}
\caption{The shared connected-sum construction and its 2-reductions.
(a) Upper row: \(\Gamma_m\); lower row: \(K_m\).
(b) Upper row: \(\Gamma_m^r\); lower row: \(K_m^r\).
The upper row uses copies of \(C_1\), and the lower row copies of \(C_2\).
After \(m\geq0\) connected-sum operations there are \(m+1\) factors,
and the reduced crossing counts are \(4(m+1)\) and \(0\), respectively.
These same diagrams are used in the based and unbased discussions.}
\label{fig:shared-family}
\end{figure}
\clearpage

\section{Unbased closed curves}\label{sec:unbased}

Let \(C\) be an oriented generic closed plane curve with no distinguished
base point. The temporary base point and the base-edge convention are fixed
in the following notation.

\begin{notation}[Base-edge convention]\label{not:base-edge}
Choose a temporary base point \(b\) in the interior of an edge of \(C\).
The edge containing \(b\) is the \emph{base edge}, denoted by \(e^b\).
If \(C\) has \(n\ge1\) crossings, number the oriented edges
\(e_1,e_2,\ldots,e_{2n}\) in the order encountered
while following the orientation of \(C\) from \(b\),
taking the base edge \(e_1=e^b\), which contains \(b\),
as the starting edge. For an edge terminating at
\(v\), we may write \(e_i^v\); in particular, the original base edge
is \(e_1^v\) when its terminal crossing is \(v\). We retain these
edges, with their intersection-number signs, as the set \(E(C)\).
At a crossing \(v\), \(e_l^v,e_r^v\) denote the two incoming edges when
the local left/right distinction is needed.
For the temporary base point \(b\), let \(T_1,T_2\) be the oriented tangent
vectors at the first and second visits to a crossing \(v\), and put
\(\sign(v,b)=-\operatorname{sgn}\det(T_1,T_2)\). The signs of the two
incoming edges at \(v\) follow the traversal order from \(b\):
\[
\begin{array}{c|cc}
  \text{first incoming edge} & s(e_l^v) & s(e_r^v)\\ \hline
  e_l^v & -1 & +1\\
  e_r^v & +1 & -1
\end{array}
\]
When \(C\) has no crossings, \(e^b\) denotes the whole circle and
\(E(C)=\varnothing\).

Introduce two symbols \(e_0^b\), \(e_1^b\), corresponding respectively
to the side entering \(b\) and the side leaving \(b\). When crossings
exist, they arise by splitting a copy of \(e^b\) at \(b\); when there
are no crossings, they are two formal generators associated with \(b\).
The contribution at the base edge is, when its terminal crossing exists,
\[
  s(e_1^v)[e_1^v]+[e_0^b]+[e_1^b],
  \qquad s(e_0^b)=s(e_1^b)=1.
\]
In both cases, set
\[
  E_b(C)=E(C)\sqcup\{e_0^b,e_1^b\},
\]
where \(\ind(e_0^b)=\ind(e_1^b)=\ind(e^b)\).
\end{notation}

\subsection{The polynomial \(\Stq\) and its low-order coefficients}\label{sec:u1}

Shumakovitch~\cite[Section~1.5]{Sh1996} gives the
following base-point-independent formula:
\[
  \St(C)=\frac12\sum_{e\in E(C)}s(e)\ind(e)^2
          +\ind(e^b)^2-\frac14 .
\]
The crossing sign \(\sign(v,b)\) follows the convention given at the
beginning of Section~\ref{sec:unbased}, equivalently in
Definition~\ref{def:based-sign}. This convention agrees with
\cite[Section~1.7.D]{Sh1996}. The edge/vertex moment
identity and base-point independence are given in
\cite[Lemmas~2.1.A and~2.2.A]{Sh1996}.
Omitting the shift \(-1/4\), define
\[
  \St^e(C)=\frac12\left(\sum_{e\in E(C)}s(e)\ind(e)^2+2\ind(e^b)^2\right)=\frac12\sum_{e\in E_b(C)}s(e)\ind(e)^2.
\]
Consequently the relation to the classical strangeness is
\[
  \St^e(C)=\St(C)+\frac14.
\]
With the base-edge convention of Notation~\ref{not:base-edge}, define
\begin{equation}\label{eq:u-polynomial}
  \Stq(C)=\sum_{e\in E_b(C)}s(e)q^{\ind(e)}
          =\sum_{e\in E(C)}s(e)q^{\ind(e)}
            +2q^{\ind(e^b)}.
\end{equation}

With the crossing-sign convention fixed above, Whitney's formula in Fukui's revised text
\cite[Section~1.6.3, Theorem~1.6.17, p.~36]{Fukui}
(cf.\ \cite[Section~1.B, p.~451]{EkholmWolff2002}, with the opposite crossing-sign convention) is then
\[
  \rot(C)=2\ind(e^b)+\sum_v\sign(v,b).
\]
At each crossing \(v\) of index \(i\), the incoming edges have indices
\(i+\tfrac12\) and \(i-\tfrac12\). Their signed first-moment
contribution is
\((i+\tfrac12)-(i-\tfrac12)\) or its negative, according to
the intersection-number convention; it equals \(\sign(v,b)\). Thus
\[
 \sum_{e\in E(C)}s(e)\ind(e)=\sum_{v\in V(C)}\sign(v,b).
\]
Combining this identity with Whitney's formula, for the edge set
\(E(C)\) we have
\[
  \sum_{e\in E(C)}s(e)=0,\qquad
  \sum_{e\in E(C)}s(e)\ind(e)=\rot(C)-2\ind(e^b).
\]
Under the base-edge convention of Notation~\ref{not:base-edge},
the corresponding sums over \(E_b(C)\) become
\begin{align}
  \sum_{e\in E_b(C)}s(e)&=2, \notag\\
  \sum_{e\in E_b(C)}s(e)\ind(e)
    &=\rot(C)-2\ind(e^b)+2\ind(e^b)=\rot(C),
      \label{eq:u-linear}\\
  \frac12\sum_{e\in E_b(C)}s(e)\ind(e)^2
    &=\St^e(C). \label{eq:u-quadratic}
\end{align}
In particular,
\[
  \St_{e^h}(C)=2+\rot(C)h+\St^e(C)h^2+O(h^3),
  \qquad \Stq(C)\big|_{q=1}=2.
\]
Thus the low-order coefficients of \(\Stq\) recover \(\rot(C)\) and \(\St^e(C)\), with constant coefficient \(2\).

\begin{lemma}[Base-point independence]\label{lem:u-basepoint-poly}
The polynomial \(\Stq\) in \eqref{eq:u-polynomial} is independent of the temporary
base point.
\end{lemma}
\begin{proof}
Suppose that \(b\) passes through a crossing of index \(i\) from the
lower-left branch to the upper-right branch. Its crossing intersection
number changes from negative to positive. The change in \(\Stq\) is
the sum of the difference in the contributions of the incident edges
and the difference in the base-point correction term:
\[
  (q^{i+1/2}-q^{i-1/2})-(q^{i-1/2}-q^{i+1/2})
  +2(q^{i-1/2}-q^{i+1/2})=0.
\]
For passage from the lower-right branch to the upper-left branch, both
differences reverse sign, giving explicitly
\[
  (q^{i-1/2}-q^{i+1/2})-(q^{i+1/2}-q^{i-1/2})
  +2(q^{i+1/2}-q^{i-1/2})=0.
\]
\end{proof}

The base-point independence of \(\St(C)+\tfrac14\) also follows from
Lemma~\ref{lem:u-basepoint-poly}: since \(\Stq(C)\) is independent of the
temporary base point, each coefficient of the Maclaurin expansion of
\(\St_{e^h}(C)\) at \(h=0\) is also base-point independent. Its quadratic
coefficient is \(\St(C)+\tfrac14\).

\subsection{Invariance and 2-reduction}\label{sec:u2}

\begin{proposition}[Uniqueness of unbased 2-reduction]\label{prop:u-unique-reduction}
Let \(C\) be an oriented generic closed plane curve without a specified
base point. Repeated crossing-decreasing self-tangency perestroikas until
no bounded 2-gons remain give a 2-reduced curve \(C^r\), uniquely determined up to orientation-preserving planar isotopy.
\end{proposition}
\begin{proof}
The proof is the same local cancellation and move-reordering argument as
for the long case (Proposition~\ref{prop:reduction} and
Lemma~\ref{lem:common-reduction}), with no distinguished base point.\footnote{The creation or removal of an unbounded 2-gon is not included in the equivalence relation for closed curves, so the listed local cases are exhaustive.}
\end{proof}

\begin{proposition}\label{prop:u-II}
The unbased polynomial is invariant under self-tangency perestroikas,
and
\[
  \Stq(C)=\Stq(C^r).
\]
\end{proposition}
\begin{proof}
By Lemma~\ref{lem:u-basepoint-poly}, move the temporary base point away
from the 2-gon before applying the move. The four incoming edges in \(E(C)\) at the two crossings have equal-index,
opposite signed contributions, which cancel. The base-edge correction is
outside the move and does not change.
\end{proof}

\subsection{Temporary base points, states, and signs}\label{sec:u3}

Temporarily base the oriented curve \(C^r\). At every crossing \(v\), if
\(e_l^v\) is encountered first and \(e_r^v\) second, set
\(s(e_l^v)=-1\) and \(s(e_r^v)=+1\); if the order is reversed, reverse the
signs. Thus
\[
  s(e_l^v)=-s(e_r^v).
\]
Every generator in
\(E_b(C^r)\setminus\{e_0^b,e_1^b\}\) has its sign determined at its terminal
crossing. The two correction generators are exceptions, with
\(s(e_0^b)=s(e_1^b)=1\).

The traversal from the temporary base point \(b\) determines the state
\(S:V(C^r)\to\{0,1\}\) of Definition~\ref{def:state}, with
\((-1)^{S(v)}=\sign(v,b)\). The upper-right local diagrams in
Figure~\ref{fig:long-marked-state} show the state together with the
oriented Seifert circles, including crossing signs, edge signs, and
wide/single-edge data. 

\subsection{Alexander numbering on \(E_b(C^r)\)}
\label{sec:u4}

For an edge in \(E(C^r)\), \(\ind(e)\) is the average of the mapping degrees
of its adjacent regions. Extend it to the correction generators by
\[
  \ind(e_0^b)=\ind(e_1^b)=\ind(e^b),
  \qquad \ind(e_1^v)=\ind(e^b).
\]

The signed sum in the free abelian group on \(E_b(C^r)\) decomposes as
\[
  \sum_{e\in E_b(C^r)}s(e)[e]
  =\sum_j\ \sum_{\substack{e\in E_b(C^r)\\\ind(e)=j}}s(e)[e].
\]

\subsection{Recovery of the polynomial and its moments}\label{sec:u5}

On the generators in \(E_b(C^r)\) for the reduced representative, the map
\([e]\mapsto q^{\ind(e)}\) gives
\[
  \sum_{e\in E_b(C^r)}s(e)[e]\longmapsto
  \sum_{e\in E_b(C^r)}s(e)q^{\ind(e)}
  =\Stq(C^r)=\Stq(C).
\]
The evaluations
\[
 [e]\longmapsto\ind(e),\qquad
 [e]\longmapsto\tfrac12\ind(e)^2
\]
recover \(\rot(C)\) and \(\St^e(C)\), respectively, by
\eqref{eq:u-linear} and \eqref{eq:u-quadratic}. The classical strangeness is
\(\St(C)=\St^e(C)-\tfrac14\).

\subsection{The edge complex of the reduced curve}\label{sec:u7}

Apply the base-edge construction to \(C^r\). Each generator in \(E_b(C^r)\)
has its assigned sign and index. Define
\[
  \mathcal E_b=\Z\langle E_b(C^r)\rangle,\qquad
  \mathcal V_b=\Z\langle V(C^r)\sqcup\{b\}\rangle .
\]
For an edge \(e\) in \(E(C^r)\), let \(v\) be its terminal crossing and write \(e=e^v\).
Set
\[
  \partial[e^v]=[v],\qquad
  \partial[e_0^b]=\partial[e_1^b]=[b],
\]
and extend linearly.
The resulting complex is
\[
  \CE_b(C^r):\qquad
  0\longrightarrow\mathcal E_b
  \xrightarrow{\partial}\mathcal V_b\longrightarrow0,
\]
where \(\mathcal E_b\) lies in homological degree \(1\) and
\(\mathcal V_b\) in degree \(0\).

At a crossing and at the temporary base point, respectively, put
\[
  c_v=[e_l^v]-[e_r^v],\qquad
  c_b=[e_0^b]-[e_1^b].
\]
These are cycles. If \(i=\ind(v)\), then locally
\[
  \ind(e_l^v)=i+\tfrac12,\qquad
  \ind(e_r^v)=i-\tfrac12,
\]
and \(\ind(e_0^b)=\ind(e_1^b)=\ind(e^b)\).

\begin{lemma}[Homology of the edge complex]\label{lem:u-rank}
For the edge complex \(\CE_b(C^r)\), including the crossing-free case,
\begin{equation}\label{eq:u-rank}
  H_0(\CE_b(C^r))=0,\qquad
  H_1(\CE_b(C^r))
    =\Z\langle c_v\ (v\in V(C^r)),\,c_b\rangle
    \cong\Z^{\,|V(C^r)|+1}.
\end{equation}
\end{lemma}
\begin{proof}
The map \(\partial\) is surjective: every crossing is hit by its
incoming edges and \(b\) is hit by the two correction generators.
Over each crossing its kernel is generated by the difference of the
two incoming generators. Over \(b\) the kernel is generated by
\([e_0^b]-[e_1^b]\). These disjoint pairs of generators give the displayed
free basis of the kernel.
\end{proof}

Define a signed \(q\)-character on the displayed canonical cycle basis
and extend it linearly. If \(i=\ind(v)\), let
\[
  \ch_q(c_v)=s(e_l^v)\bigl(q^{i+1/2}-q^{i-1/2}\bigr)=s(e_l^v)q^{\ind(e_l^v)}+s(e_r^v)q^{\ind(e_r^v)}.
\]
For \(i=\ind(e^b)\), prescribe separately
\[
  \ch_q(c_b)
  =\bigl(s(e_1^b)+s(e_0^b)\bigr)q^i
  =2q^i.
\]
Summing the character over the displayed canonical cycles gives
\begin{equation}\label{eq:u-total-character}
  \sum_{v\in V(C^r)}\ch_q(c_v)+\ch_q(c_b)
  =\sum_{e\in E_b(C^r)}s(e)q^{\ind(e)}
  =\Stq(C^r)=\Stq(C).
\end{equation}
The last equality follows from Proposition~\ref{prop:u-II}.
Define the character polynomial by
\begin{equation}\label{eq:u-edge-P}
  \mathcal P_{\CE_b}(C;t,q)
  :=t\left(\sum_{v\in V(C^r)}\ch_q(c_v)+\ch_q(c_b)\right)
  =t\Stq(C).
\end{equation}
\begin{theorem}[Unbased edge-homology lift]\label{thm:unbased-lift}
The isomorphism type of \(H_1(\CE_b(C^r))\) is independent of the temporary
base point and invariant under self-tangency perestroikas. The sum
\[
  \sum_{v\in V(C^r)}\ch_q(c_v)+\ch_q(c_b)=\Stq(C)
\]
is likewise independent of the temporary base point and invariant under
self-tangency perestroikas. The principal \(t=-1\) specialization of the character polynomial is
\[
 \mathcal P_{\CE_b}(C;-1,q)=-\Stq(C).
\]
\end{theorem}
\begin{remark}[The specialization at \(t=1\)]\label{rem:unbased-t1}
One can also use \(\mathcal P_{\CE_b}(C;1,q)=\Stq(C)\).
For the original vertex quantization \(\Stqv\) of \cite{Ito2023},
the crossing-sign convention used here gives
\(\Stqv(C)=\sum_v\sign(v,b)q^{\ind(v)}\) with a base point on
an exterior edge. Reversing the convention changes its overall sign
(cf. Remark~\ref{rem:based-signs}). In the present homological
constructions, a minus sign also arises at \(t=-1\), because the
canonical generators have homological degree \(1\).
\end{remark}
\begin{proof}[Proof of Theorem~\ref{thm:unbased-lift}]
By Lemma~\ref{lem:u-rank}, we have
\[
  H_0(\CE_b(C^r))=0,\qquad
  H_1(\CE_b(C^r))
  =\Z\langle c_v\ (v\in V(C^r)),c_b\rangle
  \cong\Z^{|V(C^r)|+1}.
\]
The unbased uniqueness of \(C^r\) (Proposition~\ref{prop:u-unique-reduction})
and Lemma~\ref{lem:u-basepoint-H} establish invariance of the group.
For the character sum, the edge state sum~\eqref{eq:u-total-character},
Lemma~\ref{lem:u-basepoint-poly} (base-point independence), and
Proposition~\ref{prop:u-II}
(invariance under self-tangency perestroikas) give the claim.
The specialization follows from \eqref{eq:u-edge-P}.
\end{proof}

\begin{lemma}[Base-point change for the underlying group]
\label{lem:u-basepoint-H}
Moving \(b\) through a crossing preserves the isomorphism type and total
rank of the underlying group \(H_1\) in \eqref{eq:u-rank}.
\end{lemma}
\begin{proof}
Map the basis generator labeled by each crossing to the generator with
the same crossing label, and map the old \(c_b\) to the new \(c_b\).
This is a basis bijection and hence an isomorphism of the free groups.
\end{proof}

For distinguishing curves, we compare the underlying ungraded free
abelian groups.

\subsection{Strictness of edge homology}\label{sec:u8}

The invariant compared in this section is the underlying group
\(H_1(\CE_b(C^r))\). Let \(C_1\) and \(C_2\) be the oriented closed
curves obtained by forgetting the base points of the left and right
curves in Figure~\ref{fig:four-crossing-pair}(a), respectively (see also
Arnold~\cite[p.~15, Fig.~12]{Arnold1994}). The comparison is
\begin{align*}
 \#V(C_1)&=\#V(C_2)=4,&\rot(C_1)&=\rot(C_2)=1,\\
 \St(C_1)&=\St(C_2)=0,&\Stq(C_1)&=\Stq(C_2)=2q^{1/2},\\
 H_1(\CE_b(C_1^r))&\not\cong H_1(\CE_b(C_2^r)).
\end{align*}
Take the temporary base point on the indicated index-\(1/2\) exterior
edge. The signed edge sum over \(E(C^r)\) is zero, and
the two base-point generators contribute \(2q^{1/2}\). Hence
\[
 \Stq(C_1)=\Stq(C_2)=2q^{1/2}.
\]
Their reduced crossing counts are four and zero, so \eqref{eq:u-rank} gives
\[
 H_1(\CE_b(C_1^r))\cong\Z^5,\qquad
 H_1(\CE_b(C_2^r))\cong\Z.
\]
For the crossing-free circle, \(E(C_2^r)=\varnothing\) and the two
base-edge generators both have index \(1/2\) and sign \(+1\).
Therefore, directly,
\[
  \Stq(C_2^r)=q^{1/2}+q^{1/2}=2q^{1/2}.
\]

Using \(C_1\) and \(C_2\) from Figure~\ref{fig:four-crossing-pair}(a)
as seeds, form the upper and lower connected-sum sequences in
Figure~\ref{fig:shared-family}(a), forgetting each base point:
\begin{align*}
  \Gamma_0&=C_1,&\Gamma_{m+1}&=\Gamma_m\# C_1,\\
  K_0&=C_2,&K_{m+1}&=K_m\# C_2\qquad(m\ge0).
\end{align*}
The comparison for these closed curves is as follows:
\begin{align*}
  \#V(\Gamma_m)&=\#V(K_m)=4(m+1),&
  \rot(\Gamma_m)&=\rot(K_m)=1,\\
  \St(\Gamma_m)&=\St(K_m)=0,&
  \Stq(\Gamma_m)&=\Stq(K_m)=2q^{1/2},\\
  H_1(\CE_b(\Gamma_m^r))&\not\cong H_1(\CE_b(K_m^r))\qquad(m\ge0).
\end{align*}
For every term, the contributions of the incoming edges in \(E(\Gamma_m^r)\) cancel
within each four-crossing factor, whereas the two base-edge generators
contribute \(2q^{1/2}\). Hence
\[
 \Stq(\Gamma_m)=\Stq(\Gamma_m^r)
 =\underbrace{\sum_{e\in E(\Gamma_m^r)}s(e)q^{\ind(e)}}_{0}+2q^{1/2}
 =2q^{1/2}.
\]
The displayed reduction has \(4(m+1)\) crossings, and therefore
\[
 H_1(\CE_b(\Gamma_m^r))\cong\Z^{4m+5}\qquad(m\geq0).
\]
These ranks are pairwise different. The lower reduced sequence in
Figure~\ref{fig:shared-family}(b) consists of crossing-free oriented
closed curves of rotation number \(1\), all equivalent under
orientation-preserving planar isotopy. Each has \(H_1\cong\Z\).
This proves strictness using edge homology and reduced crossing counts.

Finally,
\[
 2e^{h/2}=2+h+\tfrac14h^2+O(h^3),\qquad
 \rot=1,\quad \St^e=\tfrac14,\quad \St=0.
\]
Here \(\St(C)=\St^e(C)-\tfrac14\) is the classical strangeness.

\section*{Acknowledgments}
This work was supported by JSPS KAKENHI Grant Number JP25K06999.

\paragraph*{AI disclosure.}
ChatGPT (OpenAI) was used to assist with translation of the author's Japanese
notes into English and English-language editing,
expository organization, routine typesetting, and checks of mathematical
notation and internal consistency in preparing this work. It also provided
assistance in locating references and clarifying classical results.
All mathematical statements, proofs, references, and the final text were independently checked
by the author, who takes full responsibility for the work.

\bigskip
Department of Mathematics, Faculty of Engineering, Shinshu University, Wakasato 4-17-1, Nagano, Nagano 380-8553, Japan

{\texttt{nito@shinshu-u.ac.jp}}
\end{document}